%% file: SLconvexityConditions_arxiv_v1.tex
\documentclass[arxiv]{agn_article}
\usepackage[tikzexternalize=false]{agn_all}
\newcommand{\Winc}{W_{\textnormal{inc}}}
\newcommand{\Wiso}{W_{\textnormal{iso}}}
\newcommand{\Wvol}{W_{\textnormal{vol}}}
\newcommand{\GLpz}{\GLp(2)}
\newcommand{\lambdamax}{\lambda_{\mathrm{max}}}
\newcommand{\lambdamin}{\lambda_{\mathrm{min}}}
\begin{document}
\title{Rank-one convexity implies polyconvexity in isotropic planar incompressible elasticity}

\author{%
	Ionel-Dumitrel Ghiba\thanks{%
		Corresponding author: Ionel-Dumitrel Ghiba,\quad Lehrstuhl f\"{u}r Nichtlineare Analysis und Modellierung, Fakult\"{a}t f\"{u}r Mathematik,
		Universit\"{a}t Duisburg-Essen, Thea-Leymann Str. 9, 45127 Essen, Germany; Alexandru Ioan Cuza University of Ia\c si, Department of Mathematics, Blvd.~Carol I, no.~11, 700506 Ia\c si, Romania; and Octav Mayer Institute of Mathematics of the Romanian Academy, Ia\c si Branch, 700505 Ia\c si, email: dumitrel.ghiba@uni-due.de, dumitrel.ghiba@uaic.ro}%
	\ \quad and\quad%
	Robert J.\ Martin\thanks{%
		Robert J.\ Martin,\quad Lehrstuhl f\"{u}r Nichtlineare Analysis und Modellierung, Fakult\"{a}t f\"{u}r Mathematik, Universit\"{a}t Duisburg-Essen, Thea-Leymann Str. 9, 45127 Essen, Germany; email: robert.martin@uni-due.de}%
	\ \quad and\quad%
	Patrizio Neff\thanks{%
		Patrizio Neff,\quad Head of Lehrstuhl f\"{u}r Nichtlineare Analysis und Modellierung, Fakult\"{a}t f\"{u}r	Mathematik, Universit\"{a}t Duisburg-Essen, Thea-Leymann Str. 9, 45127 Essen, Germany, email: patrizio.neff@uni-due.de}%
}
\date{\today\vspace*{-1em}}
\maketitle

\begin{abstract}
We study convexity properties of energy functions in plane nonlinear elasticity of incompressible materials and show that rank-one convexity of an objective and isotropic elastic energy $W$ on the special linear group $\SL(2)$ implies the polyconvexity of $W$.
\\[1.4em]
\\[1.4em]
\end{abstract}
\tableofcontents
\newpage
\section{Introduction}
The aim of this paper is to study the relation between rank-one convexity and polyconvexity of \emph{objective} and \emph{isotropic} real valued functions $W$ on $\SL(2)= \{X\in\mathbb{R}^{2\times 2} \,\setvert\, \det X =1 \}$. These convexity properties play an important role in the theory of nonlinear hyperelasticity, where $W(\nabla\varphi)$ is interpreted as the energy density of a deformation $\varphi\colon\Omega\to\mathbb{R}^2$; here, $\Omega\subset\mathbb{R}^2$ corresponds to a planar elastic body in its reference configuration. In particular, energy functions on the domain $\SL(2)$ are used for modelling \emph{incompressible} materials, since in this case, the deformation $\varphi$ is subject to the additional constraint $\det\nabla\varphi=1$.%
\footnote{%
	Note that a function $W$ defined only on $\SL(2)$ can equivalently be expressed as a (discontinuous) function $W\colon\mathbb{R}^{2\times2}\to\mathbb{R}\cup\{+\infty\}$ with $W(F)=+\infty$ for all $F\notin\SL(2)$. This interpretation of functions not defined on all of $\R^{2\times2}$ is reflected by Mielke's definition of polyconvexity \cite{Mielke05JC} of energies $W$ on $\SL(2)$, see Definition \ref{definition:polyconvexityGLSL}.%
}

The notion of polyconvexity was introduced into the context of nonlinear elasticity theory by John Ball \cite{Ball77,Ball78} (cf.\ \cite{Raoult86,Dacorogna08,cism_book_schroeder_neff09}). Polyconvexity criteria in the case of spatial dimension $2$ were conclusively discussed by Rosakis \cite{Rosakis98} and \v{S}ilhav\'{y} \cite{Silhavy97,Silhavy99b,SilhavyPRE99,silhavy2001rank,silhavy2002monotonicity,Silhavy02b,Silhavy03}, while an exhaustive self-contained study giving necessary and sufficient conditions for polyconvexity in arbitrary spatial dimension was given by Mielke \cite{Mielke05JC}.
Rank-one convexity plays an important role in the existence and uniqueness theory for linear elastostatics and elastodynamics \cite{Ogden83,fosdick2007note,edelstein1968note,ernst1998ellipticity,knowles1976failure}. Criteria for the rank-one convexity of functions defined on $\GLp(2)=\{X\in\mathbb{R}^{2\times 2} \setvert \det X > 0\}$ were established by Knowles and Sternberg \cite{knowles1975ellipticity} as well as by \v{S}ilhav\'y \cite{SilhavyPRE99,vsilhavy2002convexity}, Dacorogna \cite{Dacorogna01}, Aubert \cite{aubert1995necessary} and Davies \cite{davies1991simple}.

It is well known that the implications 
\begin{align*}
	\text{polyconvexity} \quad\Longrightarrow\quad \text{quasiconvexity} \quad\Longrightarrow\quad \text{rank-one convexity}
\end{align*}
hold for functions on $\Rnn$ (as well as for functions on $\SL(n)$, see \cite[Theorem 1.1]{conti2008quasiconvex}) for arbitrary dimension $n$.
The reverse implications, on the other hand, do not hold in general: rank-one convexity does not imply polyconvexity \cite{alibert1992example} for dimension $n\geq2$, and rank-one convexity does not imply quasiconvexity \cite{Ball84b,Sverak92,Sverak98,Dacorogna08} for $n>2$. Whether this latter implication holds for $n=2$ is still an open question: the conjecture that rank-one convexity and quasiconvexity are \emph{not} equivalent for $n=2$ is also called \emph{Morrey's conjecture} \cite{morrey1952quasi}. For certain classes of functions on on $\R^{2\times2}$, however, it has be demonstrated that the two convexity properties are, in fact, equivalent \cite{terpstra1939darstellung,serre1983formes,marcellini1984quasiconvex,Sverak92,muller1999rank,chaudhuri2003rank,Ball84b,pedregal2014some,pedregal1996some}.

In a previous paper \cite{agn_martin2015rank}, we have shown that any energy function $W\colon\GLp(2)\to\mathbb{R}$ which is isotropic and objective (i.e.\ bi-$\OO(2)$-invariant) as well as \emph{isochoric}%
\footnote{%
	A function $W\colon\GLp(2)\to\mathbb{R}$ is called isochoric if $W(a\,F)=W(F)$ for all $a\in\mathbb{R}^+\colonequals(0,\infty)$.
	Some relations between isotropic, objective and isochoric energies and the functions defined on $\SL(2)$ are discussed in Section \ref{sectisochinc}. In elasticity theory, isochoric energy functions measure only the \emph{change of form} of an elastic body, not the \emph{change of size}.%
}
is rank-one convex if and only if it is polyconvex. In January 2016, a question by John Ball motivated some investigation into whether this result might be applicable to the incompressible case. In March 2016, at the Joint DMV and GAMM Annual Meeting in Braunschweig, Alexander Mielke indicated that some of his results \cite{Mielke05JC} should be suitable for this task.

The main result of the present paper is Theorem \ref{mainresultSL2}, which states that for objective and isotropic energies on $\SL(2)$, rank-one convexity implies (and is therefore equivalent to) polyconvexity.
Theorem \ref{mainresultSL2} includes a slightly stronger two-dimensional version of a criterion by Dunn, Fosdick and Zhang (cf.\ Section \ref{section:criteria}): an energy $W$ with $W(F)=\phi(\sqrt{\norm{F}-2})$ for $F\in \SL(2)$ is polyconvex on $\SL(2)$ (if and only if it is rank one convex) if and only if $\phi$ is nondecreasing and convex, regardless of any regularity assumption on the energy.%
\footnote{%
	Throughout this article, $\norm{X}^2=\langle {X},{X}\rangle$ denotes the Frobenius tensor norm of $X\in\Rnn$, where $\langle {X},{Y}\rangle=\tr(Y^TX)$ is the standard Euclidean scalar product on $\mathbb{R}^{n\times n}$. The identity tensor on $\mathbb{R}^{n\times n}$ will be denoted by $\id$, so that $\tr{(X)}=\langle {X},{\id}\rangle$.%
}

\section{\boldmath Rank-one convexity and polyconvexity on $\SL(2)$}
We consider the concepts of \emph{rank-one convexity} and \emph{polyconvexity} of real-valued objective, isotropic functions $W$ on the group $\GLp(2)=\{X\in\mathbb{R}^{2\times 2} \setvert \det X > 0\}$ and on its subgroup $\SL(2)=\{X\in\mathbb{R}^{2\times 2} \,\setvert\, \det X =1\}$. 
We denote by $\lambda_1, \lambda_2$ the singular values of $F$ (i.e.\ the eigenvalues of $U=\sqrt{F^T\,F}$), and $\lambdamax\colonequals\max\{\lambda_1,\lambda_2\}$ denotes the largest singular value of $F$ (also called the \emph{spectral norm} of $F$). The elastic energy $W$ is assumed to be \emph{objective} as well as \emph{isotropic}, i.e.\ to satisfy the equality
\[
	W(Q_1\,F\,Q_2) = W(F) \quad\text{ for all }\; F\in\GLpz \;\text{ and all }\; Q_1,Q_2\in\OO(2)\,,
\]
where $\OO(2)=\{X\in \mathbb{R}^{2\times 2} \setvert X^T X=\id\}$ denotes the orthogonal group.

\subsection{Basic definitions}
In order to discuss the different convexity conditions, we first need to define rank-one convexity as well as polyconvexity in the incompressible (planar) case, i.e.\ for functions on $\SL(2)$.
\subsubsection{Rank-one convexity}
Following a definition by Ball \cite[Definition 3.2]{Ball77}, we say that $W$ is \emph{rank-one convex} on $\GLp(n)$ if it is convex on all closed line segments in $\GLp(n)$ with end points differing by a matrix of rank one, i.e.
\begin{align*}
	W( F+(1-\theta)\xi\otimes \eta)=W( \theta \,F+(1-\theta)\, (F+\xi\otimes \eta))\, \leq \theta \,W( F)+(1-\theta) W(F+\xi\otimes \eta)
\end{align*}
for all $F\in \GLp(n)$, $\theta\in[0,1]$ and all $\,\, \xi,\, \eta\in\mathbb{R}^2$ with $F+t\cdot \xi\otimes \eta\in \GLp(n)$ for all $t\in[0,1]$, where $\xi\otimes\eta$ denotes the dyadic product.

Since, in the following, we will consider the case of energy functions which are defined on only on the special linear group $\SL(2)$, we need to define rank-one convexity for functions $W\colon\SL(2)\rightarrow\mathbb{R}$. The restrictions imposed by rank-one convexity are less strict in this case: the functions needs to be convex only along line segments in rank-one direction which are contained in the set $\SL(2)$, i.e.\ satisfy the additional condition
\begin{align}\label{constrain}
	\det(F+t\cdot\xi\otimes\eta)=1 \qquad\text{for all }\;t\in[0,1]\,.
\end{align}
The following lemma can be used to simplify condition \eqref{constrain} and thus allows us to give a simpler definition of rank-one convexity in the incompressible case.
\begin{lemma}
	Let $F,H\in \mathbb{R}^{2\times 2}$. Then
	\begin{align}
		\det(F+H)=\det (F) + \det (F)\,\langle F^{-T}, H\rangle +\det H\,.
	\end{align}
\end{lemma}
\begin{proof}
	For continuity reasons, it suffices to consider the case $\det F\neq 0$. Since, for $H\in\R^{2\times2}$,
	\begin{align}
		\det(\id+H)=1+\tr H+\det H\,,
	\end{align}
	we find
	\begin{align}
		\det(F+H)&=\det((\id+H\, F^{-1})\, F)=\det(\id+H\, F^{-1})\, \det(F)\nonumber\\
		&=\det(F)(1 + \tr(H\, F^{-1}) + \det(H\, F^{-1}))\nonumber\\
		&=\det(F)+\det(F)\,\langle H, F^{-T}\rangle +\det(H)\,.\nonumber\qedhere
	\end{align}
\end{proof}
\noindent Since $\rank(\xi\otimes\eta)=1$ implies $\det(\xi\otimes\eta)=0$, we thus find
\[
	\det(F+t\cdot\xi\otimes\eta) = \det F\,[1+t\,\langle F^{-T},\, \xi\otimes\eta\rangle]
\]
for $F\in \mathbb{R}^{2\times 2}$. In particular, condition \eqref{constrain} is satisfied if and only if
\begin{equation}
\label{eq:constrainDirectCondition}
	\langle F^{-T},\, \xi\otimes\eta\rangle=0\,.
\end{equation}

Condition \eqref{eq:constrainDirectCondition} can also be interpreted geometrically \cite{DunnFosdickZhang}:
if the set $\SL(2)$ is regarded as a three-dimensional surface embedded in the 4-dimensional linear space $\mathbb{R}^{2\times2}$ of all second order tensors, then the relation $D_F(\det \, F). H=(\det F)\,\langle F^{-T}, H\rangle$ implies that the tangent space $T_{\SL(2)}$ to $\SL(2)$ at $F\in \SL(2)$ is given by
\begin{align}\label{tangentspace}
	T_{\SL(2)}(F)\colonequals\{H\in \mathbb{R}^{2\times 2} \setvert \langle H, F^{-T}\rangle=0\}\,.
\end{align}
It follows that for $F\in \SL(2)$,
\begin{align}
	\det(F+t\cdot\xi\otimes\eta)=1\;\;\forall\,t\in[0,1] \quad&\iff\quad \langle \xi, F^{-T}\eta\rangle=0 \quad\iff\quad \xi\otimes \eta\in T_{\SL(2)}(F)\,.%
\end{align}
We also note that, due to the above, \eqref{constrain} already implies that the equality $\det(F+t\cdot\xi\otimes\eta)=1$ holds for all $t\in\R$ as well. These well-known (see e.g.\ \cite{DunnFosdickZhang}) equivalences allow for the following definition of rank-one convexity.
\begin{definition}
	\label{definition:rankOneConvexitySL}
	A function $W\colon\SL(2)\to\mathbb{R}$ is called \emph{rank-one convex} if the mapping
	\[
		t\mapsto W(F+t\cdot \xi\otimes \eta)%
	\]
	is convex on $\R$ for all $F\in\SL(2)$ and all $\xi,\eta\in\mathbb{R}^{2}$ such that $\xi\otimes \eta\in T_{\SL(2)}(F)$.
\end{definition}

\subsubsection{Polyconvexity}
Throughout this article, we will use the following definitions of polyconvexity for energy functions defined on the sets on $\mathbb{R}^{n\times n}$, $\GLp(n)= \{X\in\Rnn \,\setvert\, \det X >0 \}$ and $\SL(n)= \{X\in\Rnn \,\setvert\, \det X =1 \}$, respectively.
\begin{definition}\label{definition:polyconvexityGLSL}
	~
	\begin{itemize}
		\item[i)] (Ball \cite{Ball77})
			A function $W\colon\mathbb{R}^{n\times n}\to\mathbb{R}\cup\{\infty\}$ is called polyconvex if there exists a convex function $P\colon\mathbb{R}^m\to\mathbb{R}\cup\{\infty\}$ such that
			\begin{equation}
				W(F) = P(\mathbb{M}(F)) \qquad\text{for all }\;F\in\mathbb{R}^{n\times n}\,,
			\end{equation}
			where $\mathbb{M}(F)\in\mathbb{R}^m$ denotes the vector of all minors of $F$.%
		\item[ii)] (Mielke \cite{Mielke05JC})
			A function $\Winc\colon\GLp(n)\to\mathbb{R}$ is called polyconvex if the function
			\begin{align}
				\widetilde{W}\colon\mathbb{R}^{n\times n}\to\mathbb{R}\cup\{\infty\}\,,\quad \widetilde{W}(F)=
				\begin{cases}
				\Winc(F) &:\; F\in\GLp\\
				\infty &:\; F\notin\GLp
				\end{cases}
			\end{align}
			is polyconvex according to i).
		\item[iii)] (Mielke \cite{Mielke05JC})
			A function $\Winc\colon\SL(n)\to\mathbb{R}$ is called polyconvex if the function
			\begin{align}
			\widetilde{W}\colon\mathbb{R}^{n\times n}\to\mathbb{R}\cup\{\infty\}\,,\quad \widetilde{W}(F)=
			\begin{cases}
			\Winc(F) &:\; F\in\SL(n)\\
			\infty &:\; F\notin\SL(n)
			\end{cases}
			\end{align}
			is polyconvex according to i).
	\end{itemize}
\end{definition}

\subsection{Criteria for rank-one convexity and polyconvexity in the incompressible planar case}
\label{section:criteria}

For twice differentiable energies on $\SL(3)$, necessary and sufficient conditions for rank-one convexity were established by Zubov and Rudev
\cite{ZubovRudev,zubov1995necessary} as well as by Zee and Sternberg \cite{Zee83}.
An easily applicable criterion for rank-one convexity is available for the special case of differentiable functions on $\SL(3)$ of the form $F\mapsto W(F)=\phi(\gamma)$, where $\gamma=\sqrt{\norm{F}^2-3}$ represents the \emph{amount of shear} (cf.\ Theorem \ref{mainresultSL2}):%
Dunn, Fosdick and Zhang \cite{DunnFosdickZhang} have shown that the energy $W$ is rank-one convex on $\SL(3)$ if and only if $\phi$ is nondecreasing and convex. This criterion is related, with appropriate modifications, to those obtained by Zee and Sternberg \cite[p.\ 83]{Zee83}, but it only requires the energy to be once differentiable. Note that not every function on $\SL(3)$ can be written in the form $W(F)=\phi(\sqrt{\norm{F}^2-3})$, so this criterion cannot be applied in the general case of incompressible energies (as was already noted in \cite{DunnFosdickZhang}). In contrast to the three-dimensional case, every energy defined on $\SL(2)$ admits a unique representation in terms of the amount of shear $\sqrt{\norm{F}^2-2}$. This representation was also used by Mielke \cite{Mielke05JC} in order to establish necessary and sufficient conditions for polyconvexity on $\SL(2)$, cf.\ Proposition \ref{prop:mielkePolyCriterion}.

The following necessary and sufficient conditions for rank-one convexity on $\SL(2)$ are adapted from a similar criterion for the ellipticity%
\footnote{%
	Abeyaratne \cite{abeyaratne1980discontinuous} considers the \emph{ordinary ellipticity} of twice-differentiable energies, which are defined as follows:
	\vspace*{.35em}
	\begin{alignat*}{3}
		&\text{for compressible materials:} &&\det\, Q\neq 0 &&(\text{or }\; \langle Q\,\xi,\, \xi\rangle \neq 0 \quad\text{for all }\; \xi\in\mathbb{R}^2\setminus\{0\})\,,\\
		&\text{for incompressible materials:}\quad &&\det
		\matrs{Q_{11}&Q_{12}&-m_1\\ Q_{21}&Q_{22}&-m_2\\ m_1&m_2&0}
		\neq 0 \quad &&(\text{or }\; \langle Q\,\xi,\, \xi\rangle \neq 0 \quad\text{for all }\; \xi\in\mathbb{R}^2\setminus\{0\} \;\text{ with }\; \langle \xi,\, F^{-T}\eta\rangle = 0)\,,
	\end{alignat*}
	where $Q_{\alpha\gamma} \!=\! \sum_{\beta,\delta=1,2}\,\frac{\partial^2 W}{\partial F_{\alpha\beta}\,\partial F_{\gamma\delta}}\, \eta_\beta\eta_\delta$ with $\alpha,\gamma\in\{1,2\}$ is the \emph{acoustic tensor}, $m=F^{-T}\eta$ and $\eta\in\mathbb{R}^2\setminus\{0\}$. Abeyaratne's mechanical motivation is the requirement that the system of the jump equations of equilibrium should be satisfied only by the trivial solution. In the three-dimensional case, this concept was also considered by Zee and Sternberg \cite{Zee83}. We recall that for compressible materials, the \emph{strong ellipticity} (or \emph{strict Legendre-Hadamard ellipticity}) is equivalent to the positive definiteness of the acoustic tensor $Q$, while rank-one convexity is equivalent to the positive semidefiniteness of $Q$.
}
of incompressible, isotropic hyperelastic solids by Abeyaratne \cite{abeyaratne1980discontinuous}. A proof of the Proposition is given in Appendix \ref{Appendix1}.

\begin{proposition}\label{Abeyaratne}
	Let $W\colon\SL(2)\to\mathbb{R}$ be a twice-differentiable objective and isotropic function. Then there exists a unique function $\psi\colon[0,\infty)\to\mathbb{R}$ such that
	\begin{equation}
	\label{eq:representationFormulaHat}
		W(F) = \psi(I),\qquad I=\norm{F}^2=\lambdamax^2(F)+\frac{1}{\lambdamax^2(F)}
	\end{equation}
	for all $F\in\SL(2)$, where $\lambdamax(F)$ is the largest singular value of $F$. Furthermore, the following are equivalent:
	\begin{itemize}
		\item[i)] $W$ is rank-one convex,
		\item[ii)] $\psi$ satisfies the inequalities 
			\begin{align}
			\frac{d \psi}{d I}(I)\geq 0\,, \quad 2\, (I-2)\,\frac{d^2 \psi}{d I^2}(I)+\frac{d \psi}{d I}(I)\geq 0 \quad\text{for all }\; I\in [2,\infty)\,.
			\end{align}
	\end{itemize}
\end{proposition}

In the two-dimensional incompressible case, i.e.\ for an objective and isotropic energy $W$ on $\SL(2)$, another representation of the energy can be obtained from formula \eqref{eq:representationFormulaHat}: since, for $F\in\SL(2)$,
\[
	\gamma\colonequals\sqrt{\norm{F}^2-2} = \sqrt{\lambda_1^2+\lambda_2^2-2} = \sqrt{(\lambda_1-\lambda_2)^2} = \abs{\lambda_1-\lambda_2} = \lambdamax-\lambdamin = \lambdamax-\frac{1}{\lambdamax}
\]
and
\[
	I=2+\left(\lambdamax-\frac{1}{\lambdamax}\right)^2=2+\gamma^2\,,
\]
there exists a unique function $\phi\colon[0,\infty)\to \mathbb{R}$ such that
\begin{equation}
	W(F)= \psi(I) = \phi(\sqrt{\norm{F}^2-2}) = \phi\Big(\lambdamax-\frac{1}{\lambdamax}\Big) \quad\text{ for all }\; F\in \SL(2)\,.
\end{equation}
The next criterion for rank-one convexity in terms of this representation can be obtained by a direct adaptation of the proof of the aforementioned three-dimensional result by Dunn, Fosdick and Zhang \cite{DunnFosdickZhang} to the two-dimensional case.
\begin{proposition}
	\label{propFosdick}
	Let $W\colon\SL(2)\to\mathbb{R}$ be an objective and isotropic differentiable function. Then there exists a unique function $\phi\colon[0,\infty)\to\mathbb{R}$ such that
	\[
		W(F) =\phi\Big(\lambdamax(F)-\frac{1}{\lambdamax(F)}\Big)
	\]
	for all $F\in\SL(2)$, where $\lambdamax(F)$ is the largest singular value of $F$. Furthermore, the following are equivalent:
	\begin{itemize}
		\item[i)] $W$ is rank-one convex,
		\item[ii)] $\phi$ is nondecreasing and convex on $[0,\infty)$.
	\end{itemize}
\end{proposition}

It is easy to see that if an energy (and thus $\psi$) is twice differentiable, then Theorem \ref{propFosdick} and Theorem \ref{Abeyaratne} are equivalent: for $\gamma=\sqrt{I-2}$ we find
\begin{align}
	\frac{d \psi}{d I}=\frac{d \phi}{d \gamma}\,\frac{1}{2\,\sqrt{I-2}} = \frac{1}{\,\gamma}\,\frac{d \phi}{d \gamma}
	\qquad\text{and}\qquad
	\frac{d^2 \psi}{d I^2} = \frac{1}{4\,\gamma^2}\,\frac{d^2 \phi}{d \gamma^2} - \frac{1}{4\,\gamma^3}\,\frac{d \phi}{d \gamma}\,,
\end{align}
thus the monotonicity of $\phi$ is equivalent to $\frac{d \psi}{d I}(I)\geq 0$ for all $I\in [2,\infty)$, while the convexity of $\phi$ is equivalent to $2\,(I-2)\,\frac{d^2 \psi}{d I^2}(I)+\frac{d \psi}{d I}(I)\geq 0$ for all $I\in [2,\infty)$.

\vspace*{.7em}

In addition to these criteria for rank-one convexity, we will use the following polyconvexity criterion, which is due to Mielke \cite[Theorem 5.1]{Mielke05JC}.
\begin{proposition}
	\label{prop:mielkePolyCriterion}
	Let $W\colon\SL(2)\to\mathbb{R}$ be an objective and isotropic function, and $\phi\colon[0,\infty)\to\mathbb{R}$ the unique function with
	\[
	W(F) = \phi\Big(\lambdamax(F)-\frac{1}{\lambdamax(F)}\Big)
	\]
	for all $F\in\SL(2)$, where $\lambdamax(F)$ is the largest singular value of $F$. The following are equivalent:
	\begin{itemize}
		\item[i)] $\phi$ is nondecreasing and convex on $[0,\infty)$,
		\item[ii)] $W$ is polyconvex (in the sense of Definition \ref{definition:polyconvexityGLSL} iii)).
	\end{itemize}
\end{proposition}

\section{\boldmath Equivalence of rank-one convexity and polyconvexity on $\SL(2)$}
\label{section:mainResult}

We now want to show that for objective and isotropic energy functions on $\SL(2)$, rank-one convexity and polyconvexity are equivalent.
\subsection{Differentiable functions}
For \emph{differentiable} functions, this result can be obtained directly by comparing Proposition \ref{propFosdick} and Proposition \ref{prop:mielkePolyCriterion}.
\begin{proposition}
	\label{equivalenceDiffable}
	Let $W\colon\SL(2)\to\mathbb{R}$ be objective and isotropic as well as \emph{differentiable}. Then the following are equivalent:
	\begin{itemize}
		\item[i)] $W$ is rank-one convex,
		\item[ii)] the function $\phi\colon[0,\infty)\to\mathbb{R}$ with $W(F) =\phi\Big(\lambdamax(F)-\displaystyle\frac{1}{\lambdamax(F)}\Big)$ is nondecreasing and convex,
		\item[iii)] $W$ is polyconvex.
	\end{itemize}
\end{proposition}

\subsection{The general case}
Our main result of this paper is the equivalence of rank-one convexity and polyconvexity for objective and isotropic energy functions in general, without any regularity assumptions.
The following theorem also provides another geometric interpretation of the criteria from Proposition \ref{propFosdick} and Proposition \ref{prop:mielkePolyCriterion} by Dunn et al.\ and Mielke: the convexity and monotonicity of the function $\phi$ is equivalent to the convexity of the energy with respect to the \emph{amount of shear}. Note that in the planar incompressible case, an energy function is already completely determined by its response to simple shear deformations \cite{abeyaratne1980discontinuous}.\footnote{%
Any plane volume preserving deformation can be decomposed locally into the product of a simple shear in a suitable direction followed or preceded by a suitable rotation. More precisely, for any $F\in \SL(2)$, there exist $Q_1,Q_2\in \OO(2)$ such that
\begin{align}
	F=Q_1\, K\,Q_2\,, \qquad K = \matr{1&\gamma\\0&1}\,, \qquad \gamma=\pm\sqrt{I-2}\,, \qquad I=\norm{F}^2\,.
\end{align}
}
\begin{theorem}\label{mainresultSL2}
	Let $W\colon\SL(2)\to\mathbb{R}$ be an objective and isotropic function. Then the following are equivalent:
	\begin{itemize}
		\item[i)] $W$ is rank-one convex,
		\item[ii)] $W$ is polyconvex,
		\item[iii)] the mapping $\widetilde{\phi}\colon\mathbb{R}\to\mathbb{R}\,,\;\; \widetilde{\phi}(\gamma) = W(\matrs{1&\gamma\\0&1})$ is convex,
		\item[iv)] the function $\phi\colon[0,\infty)\to\mathbb{R}$ with $W(F) = \phi(\sqrt{\norm{F}^2-2})=\phi\Big(\lambdamax(F)-\displaystyle\frac{1}{\lambdamax(F)}\Big)$
		is nondecreasing and convex.
	\end{itemize}
\end{theorem}
\begin{proof}~\\
	i) $\implies$ iii):\\
	We note that
	\[
	\matr{1&\gamma\\0&1} = \id + \gamma\cdot \xi\otimes \eta \qquad\text{ with }\quad
	\xi= \begin{pmatrix}
		1\\0
	\end{pmatrix}\,, \quad
	\eta=\begin{pmatrix}
		0\\1
	\end{pmatrix}\,.
	\]
	 Furthermore, $\det(\id+\gamma\cdot \xi\otimes \eta)=1$ for all $\gamma\in\mathbb{R}$. Thus the rank-one convexity of $W$ implies that the mapping $\gamma\mapsto W(\matrs{1&\gamma\\0&1}) = W(\id+\gamma\cdot \xi\otimes \eta)$ is convex on $\mathbb{R}$.
	
	\vspace*{.7em}\noindent iii) $\implies$ iv):\\
	Let $\phi\colon[0,\infty)\to\mathbb{R}$ denote the uniquely defined function with $W(F)=\phi(\lambdamax(F)-\tfrac{1}{\lambdamax(F)})$ for all $F\in\SL(2)$. We first show that $\phi(t)=\widetilde{\phi}(t)$ for all $t\geq0$: for $\gamma\geq0$, the singular values of the simple shear are
	\begin{align}
		\lambdamax(\matr{1&\gamma\\0&1}) = \frac12\,(\gamma+\sqrt{\gamma^2+4}) \qquad\text{and}\qquad \lambdamin(\matr{1&\gamma\\0&1}) = \frac12\,(-\gamma+\sqrt{\gamma^2+4})\,.
	\end{align}
	Thus we find
	\begin{align}
		\lambdamax(\matr{1&\gamma\\0&1}) - \left(\lambdamax(\matr{1&\gamma\\0&1})\right)^{-1}& = \lambdamax(\matr{1&\gamma\\0&1}) - \lambdamin(\matr{1&\gamma\\0&1})\nonumber\\
		&= \frac{1}{2}\,(\gamma+\sqrt{\gamma^2+4}) - \frac{1}{2}\,(-\gamma+\sqrt{\gamma^2+4}) = \gamma
	\end{align}
	and therefore
	\[
		\widetilde{\phi}(\gamma) = W(\matr{1&\gamma\\0&1}) = \phi\left(\lambdamax(\matr{1&\gamma\\0&1}) - \left(\lambdamax(\matr{1&\gamma\\0&1})\right)^{-1}\right) = \phi(\gamma)
	\]
	for all $\gamma\geq0$.
	
	\vspace*{.35em}
	Since $\widetilde{\phi}$ is convex by assumption of condition iii), it follows that $\phi=\widetilde{\phi}\big|_{[0,\infty)}$ is convex on $[0,\infty)$ as well. Thus it only remains to show that $h=\widetilde{\phi}\big|_{[0,\infty)}$ is also nondecreasing.
	
	Let $0\leq t_1 \leq t_2$. Then $t_1$ lies in the convex hull of $-t_2$ and $t_2$, i.e.\ $t_1 = s(-t_2)+(1-s)\,t_2$ for some $s\in[0,1]$. Since $\widetilde{\phi}$ is convex on $\R$ and $\widetilde{\phi}(-t)=\widetilde{\phi}(t)$ for all $t\in\R$, we thus find %
	\[
		\phi(t_1) = \widetilde{\phi}(t_1)\leq s\,\widetilde{\phi}(-t_2)+(1-s)\,\widetilde{\phi}(t_2) = \widetilde{\phi}(t_2) = \phi(t_2)\,,
	\]
	which shows that $\phi$ is nondecreasing.
	
	\vspace*{.7em}\noindent iv) $\implies$ ii):\\
	Condition iv), i.e.\ the convexity and monotonicity of $\phi$, is exactly condition i) in Proposition \ref{prop:mielkePolyCriterion}, which immediately implies that $W$ is polyconvex. Note that in contrast to Proposition \ref{propFosdick}, Proposition \ref{prop:mielkePolyCriterion} does not require the energy to be differentiable.
	
	\vspace*{.7em}\noindent ii) $\implies$ i):\\
	This implication is well known, see for example \cite[Theorem 5.3]{Dacorogna08}.
\end{proof}
\begin{remark}
	In addition to showing the equivalence between rank-one convexity and polyconvexity on $\SL(2)$, Theorem \ref{mainresultSL2} requires no regularity of the energy and thus improves the known criteria for rank-one convexity on $\SL(2)$.
\end{remark}

\section{\boldmath Functions on $\SL(2)$ and isochoric functions on $\GLp(2)$}
\label{sectisochinc}

Functions on the special linear group $\SL(2)$ are closely connected to so-called \emph{isochoric} function on $\GLp(2)$, i.e.\ functions $\Wiso\col\GLp(2)\to\R$ with $\Wiso(a\,F)=W(F)$ for all $a\in\mathbb{R}^+$. In particular, any isochoric function can be written as \cite{agn_martin2015rank}
\begin{equation}
\label{eq:incompressibleIsochoricRelation}
	\Wiso(F) = \Winc\left( \frac{F}{(\det F)^{1/2}} \right)\,,
\end{equation}
where $\Winc=W\big|_{\SL(2)}$ is the restriction of $\Wiso$ to $\SL(2)$. Furthermore, the relation \eqref{eq:incompressibleIsochoricRelation} describes a \emph{bijection} between the set of isochoric functions and the set of functions on $\SL(2)$. We also note that $\Winc$ is objective/isotropic if and only $\Wiso$ is objective/isotropic.

A result similar to Theorem \ref{mainresultSL2} has previously been shown to hold for isochoric functions \cite{agn_martin2015rank}. In the following, we briefly discuss a failed first attempt to prove Thoerem \ref{mainresultSL2} by using this earlier result, thereby highlighting the difference between convexity properties of isochoric functions and functions on $\SL(2)$.

\begin{proposition}[\cite{agn_martin2015rank}]
	\label{prop:mainResultIsochoric}
	Let $\Wiso\colon\GLp(2)\to\mathbb{R}$ be an objective, isotropic and isochoric function, i.e.
	\[
		\Wiso(a\,F)=\Wiso(F) \quad\text{ for all }\; a\in\mathbb{R}^+\colonequals(0,\infty)\,,
	\]
	and let $g\colon\mathbb{R}^+\times\mathbb{R}^+\to\mathbb{R}\,,\; h\colon\mathbb{R}^+\to\mathbb{R}$ denote the uniquely determined functions with
	\begin{align}\label{defh}
		\Wiso(F)=g(\lambda_1,\lambda_2)=h\left(\frac{\lambda_1}{\lambda_2}\right)=h\left(\frac{\lambda_2}{\lambda_1}\right)
	\end{align}
	for all $F\in\GLp(2)$ with singular values $\lambda_1,\lambda_2$. Then the following are equivalent:
	\begin{itemize}
		\item[i)] $\Wiso$ is polyconvex.
		\item[ii)] $\Wiso$ is rank-one convex,
		\item[iii)] $g$ is separately convex,
		\item[iv)] $h$ is convex on $\mathbb{R}^+$,
		\item[v)] $h$ is convex and non-decreasing on $[1,\infty)$.
	\end{itemize}
\end{proposition}

Of course, in order to show the equivalence of rank-one convexity and polyconvexity for functions on $\SL(2)$, one might attempt to combine Proposition \ref{prop:mainResultIsochoric} with the relation \eqref{eq:incompressibleIsochoricRelation}.

This approach, which is visualized in Fig.\ \ref{figure:SLapproachVisualization}, can be summarized as follows: we would like to show that the rank-one convexity of $\Winc$ implies the rank-one convexity of $\Wiso$ (in the notation of \eqref{eq:incompressibleIsochoricRelation}). If this was the case, then we could apply Theorem \ref{prop:mainResultIsochoric} to show that $\Wiso$ is polyconvex, and thus $\Winc$ is polyconvex as the restriction of the polyconvex function $\Wiso$ to $\SL(2)$, cf.\ Definition \ref{definition:polyconvexityGLSL}.

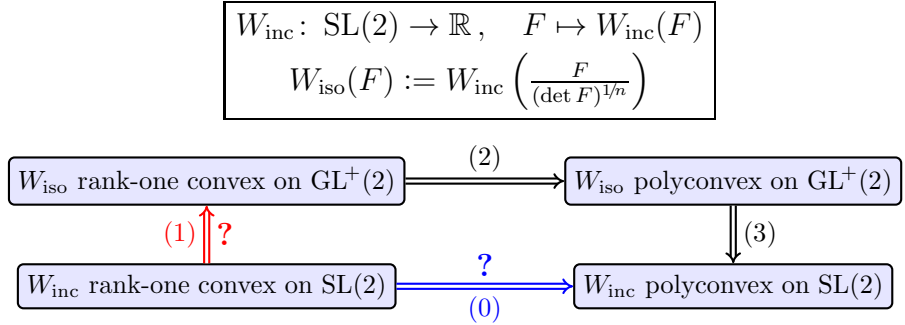
\begin{figure}[h]
	\begin{center}
		\tikzsetnextfilename{SLconvexityApproachVisualization}
		\begin{tikzpicture}
			\input{SLconvexityApproachVisualization}
		\end{tikzpicture}
	\end{center}
	\caption{\label{figure:SLapproachVisualization} A visualization of our first approach: implications (2) and (3) hold (see Proposition \ref{remarkisoSL} and Proposition \ref{prop:mainResultIsochoric}), whereas it turns out from Remark \ref{remarkcounter} that implication \textcolor{red}{(1)} does not hold in general. Implication \textcolor{blue}{(0)} is the main result of this article (Theorem \ref{mainresultSL2})}
\end{figure}

However, this approach turned out not to be viable: although the rank-one convexity of $\Wiso$ implies the rank-one convexity of $\Winc$, the reverse is not true in general.

\begin{proposition}
\label{remarkisoSL}
	Let $\Wiso\colon\GLp(2)\to\mathbb{R}$ be an objective, isotropic and isochoric function. Then rank-one convexity (equivalently polyconvexity) of $\Wiso$ on $\GLp(2)$ implies rank-one convexity (equivalently polyconvexity) of $\Winc\colon\SL(2)\to\mathbb{R}$ on $\SL(2)$. The reverse implication does not hold in general.
\end{proposition}
\begin{proof}
	Since $\Wiso\colon\GLp(2)\to\mathbb{R}$ is an objective, isotropic, isochoric rank-one convex function on $\GLp(2)$, the unique function \hbox{$h\colon\mathbb{R}^+\to\mathbb{R}$\;} satisfying \eqref{defh} is convex and non-decreasing on $[1,\infty)$. For all $F\in \SL(2)$,
	\begin{align}\label{isoinc}
		h\left(\lambdamax^2(F)\right)=h\left(\frac{\lambdamax(F)}{\lambdamin(F)}\right)=\Wiso(F)=\Winc(F) =\phi\Big(\lambdamax(F)-\frac{1}{\lambdamax(F)}\Big)\,,
	\end{align}
	where $\lambdamax(F)$ is the largest singular value of $F$ and $\phi\colon[0,\infty)\to\mathbb{R}$ is the unique function such that the last equality of \eqref{isoinc} holds. Therefore
	\begin{align}
		\phi(\theta)=h\left(\frac{(\theta+\sqrt{\theta+4})^2}{4}\right) \quad\text{ for all }\; \theta\geq 0\,.
	\end{align}
	Since the mapping $\theta\mapsto \frac{(\theta+\sqrt{\theta+4})^2}{4}$ is a convex from $[0,\infty)$ to $[1,\infty)$  and \hbox{$h\colon\mathbb{R}^+\to\mathbb{R}$\;} is convex and non-decreasing on $[1,\infty)$, the function $\phi\colon[0,\infty)\to\mathbb{R}$ is nondecreasing and convex. Thus Theorem \ref{mainresultSL2} yields the polyconvexity and rank-one convexity of the function $\Winc=W\big|_{\SL(2)}$. For the second part of the proof, we refer to Remark \ref{remarkcounter} for a counterexample.
\end{proof}

\begin{remark}[Counterexample to the above approach]
\label{remarkcounter}
	Consider the function $\Wiso\colon\GLp(2)\to\mathbb{R}$ with
	\[
		\Wiso(F) = \bigg\lvert\sqrt{\frac{\lambda_1}{\lambda_2}} - \sqrt{\frac{\lambda_2}{\lambda_1}}\;\bigg\rvert
	\]
	for all $F\in\GLp(2)$ with singular values $\lambda_1,\lambda_2\in\mathbb{R}^+$. Then
	\begin{itemize}
		\item[i)] $\Wiso$ is objective, isotropic and isochoric on $\GLp(2)$,
		\item[ii)] $\Wiso$ is \emph{not} rank-one convex on $\GLp(2)$,
		\item[iii)] the restriction $\Winc=\Wiso\big|_{\SL(2)}$ of $\Wiso$ to $\SL(2)$ is polyconvex and rank-one convex on $\SL(2)$.
	\end{itemize}
\end{remark}
\begin{proof}
	In order to show i), it suffices to remark that
	\[
		\Wiso(F) = h\left(\frac{\lambda_1}{\lambda_2}\right) \quad\text{ with }\quad h(t) = \bigg\lvert\sqrt{t}-\sqrt{\frac1t}\;\bigg\rvert
	\]
	for all $F\in\GLp(2)$ with singular values $\lambda_1,\lambda_2$. Thus $\Wiso$ is objective, isotropic and isochoric, and according to Theorem \ref{prop:mainResultIsochoric}, $\Wiso$ is rank-one convex if and only if $h$ is convex and non-decreasing on $[1,\infty)$. Since
	\[
		h''(t) = -\frac14\,t^{-\frac32} - \frac34\, t^{-\frac52}
	\]
	for all $t>1$, $h$ is not convex, which proves ii).
	
	It remains to show iii), i.e.\ that the restriction $\Winc=\Wiso\big|_{\SL(2)}$ of $\Wiso$ to $\SL(2)$ is ($\SL$-)polyconvex. We first give an explicit representation of $\Winc$: let $F\in\SL(2)$. Then $\frac{1}{\lambda_2}=\lambda_1\equalscolon\lambda$, thus
	\begin{align*}
		\Winc(F) &= \Wiso(F)= \bigg\lvert\sqrt{\frac{\lambda_1}{\lambda_2}} - \sqrt{\frac{\lambda_2}{\lambda_1}}\;\bigg\rvert= \Bigg\lvert\sqrt{\frac{\lambda}{\frac1\lambda}} - \sqrt{\frac{\frac1\lambda}{\lambda}}\;\Bigg\rvert= \bigg\lvert\sqrt{\lambda^2} - \sqrt{\frac{1}{\lambda^2}}\;\bigg\rvert= \Big\lvert\lambda - \frac1\lambda\Big\rvert\\
		&= \max\{\lambda,\tfrac1\lambda\} - \min\{\lambda,\tfrac1\lambda\}
		= \max\{\lambda_1,\lambda_2\} - \min\{\lambda_1,\lambda_2\}
		= \lambdamax - \frac{1}{\lambdamax}
		= \phi(\lambdamax-\tfrac{1}{\lambdamax})\,,
	\end{align*}
	where $\lambdamax = \max\{\lambda_1,\lambda_2\} = \max\{\lambda,\tfrac1\lambda\}$ and $\phi\colon[0,\infty)\to\mathbb{R}$ is defined by $\phi(s) = s$. According to Proposition \ref{prop:mielkePolyCriterion}, a function $W\colon\SL(2)\to\mathbb{R}$ with $W(F) = \phi(\lambdamax-\tfrac{1}{\lambdamax})$ for all $F\in\SL(2)$ with maximum singular value $\lambdamax$ is polyconvex if and only if $\phi$ is non-decreasing and convex on $[0,\infty)$. Since the mapping $\phi$ obviously satisfies these conditions, the function $\Winc$ is polyconvex on $\SL(2)$ and thus rank-one convex \cite[Theorem 5.3]{Dacorogna08} (cf.\ Theorem \ref{mainresultSL2}).
\end{proof}

\section{Outlook}

We finish with some open questions: consider an energy $W\colon\GLp(2)\to \mathbb{R}$ with volumetric-isochoric split%
\footnote{%
	The applications of a split of the form \eqref{eq:volIsoSplit} in the anisotropic case has been discussed by Sansour \cite{sansour2008physical}. Sansour's statement for isotropy is already contained in an 1948 article by Richter \cite[page 209]{richter1948isotrope}.
}%
\begin{equation}
\label{eq:volIsoSplit}
	W(F)=\Wiso\bigg(\frac{F}{\det F^{1/2}}\bigg)+\Wvol(\det F)\,.
\end{equation}
Such a split is relevant for slightly compressible materials like vulcanized rubber, cf.\ \cite{annaidh2012deficiencies,horgan2009constitutive,NeffGhibaLankeit,NeffGhibaPoly,ndanou2014criterion}. It is an open question whether rank-one convexity implies polyconvexity for an energy $W$ of this type. There might also be additional restrictions on the isochoric part $\Wiso(\frac{F}{\det F^{1/2}})$ and the volumetric part $\Wvol(\det F)$ which assure that the implication holds.

A related question is whether the rank-one convexity of the total energy $W$ implies the rank-one convexity of the individual parts $\Wiso(\frac{F}{\det F^{1/2}})$ and $\Wvol(\det F)$, respectively. In particular, this would imply (due to Theorem \ref{mainresultSL2} and Proposition \ref{prop:mainResultIsochoric}) that rank-one convexity is equivalent to polyconvexity for energy functions of the form \eqref{eq:volIsoSplit}.

\section*{Acknowledgement}
The interest in studying the equivalence between rank-one convexity and polyconvexity on $\SL(2)$ arose after an insightful question by Sir John Ball (Oxford Centre for Nonlinear Partial Differential Equations) regarding our previous paper \cite{agn_martin2015rank}. Discussions with Prof.\ Alexander Mielke (Weierstrass Institute for Applied Analysis and Stochastics, Berlin) regarding this topic are also gratefully acknowledged.

\begin{footnotesize}
\printbibliography[heading=bibnumbered]
\end{footnotesize}

\newpage
\appendix 
{\footnotesize

\section{\boldmath An alternative proof of a rank-one convexity criterion for twice-differentiable functions on $\SL(2)$}
\label{Appendix1}

In this appendix, we will give an alternative proof of the rank-one convexity criterion stated in Proposition \ref{Abeyaratne}.
We assume that the energy $W\col\SL(2)\to\R$ is twice-differentiable. In this case, rank-one convexity is equivalent to \emph{Legendre-Hadamard ellipticity} on $\SL(2)$:
\begin{equation}
	D^2_F W(F)(\xi\otimes\eta,\xi\otimes\eta)\geq0 \quad\text{ for all }\; F\in \SL(2) \;\text{ and }\; \xi,\eta\in\mathbb{R}^2 \;\text{ with }\; \xi\otimes \eta\in T_{\SL(2)}(F) \,.
\end{equation}
The Legendre-Hadamard ellipticity condition can equivalently be stated as
\begin{align}\label{LH2}
	\langle Q(F,\eta)\, \xi,\, \xi\rangle \geq 0 \quad\text{ for all }\;F\in \SL(2) \;\text{ and }\; \xi,\eta\in\mathbb{R}^2 \;\text{ such that }\; \langle F^{-T} \eta,\, \xi\rangle=0\,,
\end{align}
where the acoustic tensor $Q=(Q_{\alpha\gamma})_{{\alpha\gamma}}$ is defined by
\begin{align}
	Q_{\alpha\gamma}(F,\eta) = \frac{\partial^2 W(F)}{\partial F_{\alpha\beta}\, \partial F_{\gamma\delta}}\, \eta_\beta\eta_\delta\,.
\end{align}
Here, we employ the Einstein summation convention for Greek subscripts (which take the values $1,2$).

Note that for $\eta=0$, the Legendre-Hadamard ellipticity condition is satisfied for all $F\in \SL(2)$ and all $\xi\in \mathbb{R}^2$.
Hence we may assume that $\eta$ is a unit vector. Note also that for all $ \eta\in\mathbb{R}^2\setminus\{0\}$ with $\norm{\eta}=1$ and all $F\in \SL(2)$,
\begin{align}
	\langle F^{-T} \eta,\, \xi\rangle=0 \quad\iff\quad \xi=a\, \epsilon\, F^{-T}\eta \quad\text{for some }\; \ a\in \mathbb{R}\,,
\end{align}
where
\begin{align}
	\epsilon \colonequals \matr{0&1\\-1&0}
\end{align}
is the two-dimensional alternator. Since $F^{-1}=\epsilon^T\, F\, \epsilon$ for $F\in \SL(2)$ and $\epsilon\, \epsilon^T=\id$, we find that condition \eqref{LH2} is equivalent to
\begin{align}\label{LHps}
	\langle Q(F,\eta)\, F^T\, \epsilon\,\eta, F^T \epsilon\, \; \eta\rangle \geq 0 \quad\text{for all }\; F\in \SL(2) \;\text{ and }\; \eta\in\mathbb{R}^2\setminus\{0\} \;\text{ with }\; \norm{\eta}=1\,,
\end{align}
which can be written in terms of the components as
\begin{align}
	\epsilon_{\alpha\lambda}\epsilon_{\beta\mu}F_{\gamma\lambda}F_{\delta \mu}Q_{\gamma\delta}\eta_\alpha\eta_\beta\geq 0 \quad\text{for all }\; F\in \SL(2) \;\text{ and }\; \eta\in\mathbb{R}^2\setminus\{0\} \;\text{ with }\; \norm{\eta}=1\,.
\end{align}
Now consider the fourth order elasticity tensor defined by
\begin{align}
	\mathbb{C}_{\alpha\beta\gamma\delta}=\frac{\partial^2 W(F)}{\partial F_{\alpha\beta}\partial F_{\gamma\delta}}\,.
\end{align}
The Legendre-Hadamard condition on $\SL(2)$ is equivalent to
\begin{align}\label{LHpsc2}
	\langle \mathbb{C}.\, ( F^T\, \epsilon\,\eta)\otimes \eta,\; ( F^T\, \epsilon\,\eta)\otimes \eta\rangle \geq 0 \quad\text{for all }\; F\in \SL(2) \;\text{ and }\; \eta\in\mathbb{R}^2\setminus\{0\} \;\text{ with }\; \norm{\eta}=1\,.
\end{align}
For objective and isotropic energies $W\colon\SL(2)\to \mathbb{R}$ given by
\begin{align}
	W(F) = \psi(I),\qquad I=\norm{F}^2=\lambdamax^2(F)+\frac{1}{\lambdamax^2(F)}\,,%
\end{align}
we find
\begin{align}
	\frac{\partial W(F)}{\partial F_{\alpha\beta}}=2\, F_{\alpha\beta}\, \psi^{\prime}(I)\,,
	\qquad
	\mathbb{C}_{\alpha\beta\gamma\delta} = 4\, F_{\alpha\beta}\, F_{\gamma\delta}\, \psi^{\prime\prime}(I)+2\, \delta_{\alpha\gamma}\,\delta_{\beta\delta} \,\psi^{\prime}(I)\,,
\end{align}
and the acoustic tensor is given by 
\begin{align}
	Q_{\alpha\gamma}=\mathbb{C}_{\alpha\beta\gamma\delta}\;\eta_\beta\eta_\delta &= 4\, F_{\alpha\beta}\, F_{\gamma\delta}\,\eta_\beta\,\eta_\delta\, \psi^{\prime\prime}(I)+2\, \delta_{\alpha\gamma}\,\delta_{\beta\delta} \,\eta_\beta\,\eta_\delta\,\psi^{\prime}(I)\nonumber\\
	&= 4\, F_{\alpha\beta}\,\eta_\beta\, F_{\gamma\delta}\eta_\delta\, \psi^{\prime\prime}(I)+2\, \delta_{\alpha\gamma}\,\eta_\delta\,\eta_\delta\,\psi^{\prime}(I).
\end{align}
Hence, after some calculations (cf.\ \cite{abeyaratne1980discontinuous}), condition \eqref{LHps} becomes
\[
	(\epsilon_{\alpha\lambda}\epsilon_{\beta\mu}C_{\alpha\beta}\; \eta_\lambda\eta_\mu)\, \psi^{\prime}(I)+2(\epsilon_{\alpha\lambda}C_{\alpha \rho}\; \eta_\rho\eta_\lambda)^2 \psi^{\prime\prime}(I)\geq 0 \quad\text{ for all }\; F\in \SL(2)\,,\;\; \eta\in\mathbb{R}^2\setminus\{0\}\,,\; \norm{\eta}=1\,,
\]
where $C=F^T F$. Due to isotropy, we can assume without loss of generality that
\begin{align}
	C=\matr{\lambda_1^2&0\\0&\lambda_2^2},\,
\end{align}
thus we obtain the condition
\begin{alignat}{2}
\label{LHlambda}
	&(\lambda_1^2\eta_2^2+\lambda_2^2\,\eta_1^2)^2\,\, \psi^{\prime}(I) + 2(\lambda_1^2-\lambda_2^2)^2\;\eta_1^2\eta^2_2\, \psi^{\prime\prime}(I)\geq 0 \quad&&\text{ for all }\; F\in \SL(2)\,,\;\; \eta\in\mathbb{R}^2\setminus\{0\}\,,\;\; \norm{\eta}=1\,,\\
\intertext{which in turn is equivalent to}
	&\mathrlap{[\lambda_2^2\, \psi^\prime(I)]\, \eta_1^4+[\lambda_1^2\, \psi^\prime(I)]\, \eta_2^4+[(\lambda_1^2+\lambda_2^2)\psi^\prime(I)+2(\lambda_1^2-\lambda_2^2)^2\,\psi^{\prime\prime}(I)]\,\eta_1^2\,\eta_2^2 \,\geq\, 0}&&\\
	&&&\text{ for all }\; F\in \SL(2)\,,\;\; \eta\in\mathbb{R}^2\setminus\{0\}\,,\;\; \norm{\eta}=1\,.\nonumber
\end{alignat}
Using the notation \cite{abeyaratne1980discontinuous}
\[
	E_{11}=\lambda_2^2\, \psi^\prime(I), \quad E_{22}=\lambda_1^2\, \psi^\prime(I),\quad E_{12}=E_{21}=\frac{1}{2}[(\lambda_1^2+\lambda_2^2)\psi^\prime(I)+2(\lambda_1^2-\lambda_2^2)^2\,\psi^{\prime\prime}(I)]\,,
\]
condition \eqref{LHlambda} can be written as
\begin{equation}
\label{eq:almostPosSemiDef}
	\langle E\, \zeta,\, \zeta\rangle\geq 0 \quad\text{ for all }\; \zeta\in\mathbb{R}^2_+\,.
\end{equation}
Note carefully that \eqref{eq:almostPosSemiDef} does not imply that $E$ is positive semi-definite, since the inequality needs to hold only for $\zeta\in\mathbb{R}^2_+$ and not for all $ \zeta\in\mathbb{R}^2$.
Instead, it is easy to see that condition \eqref{eq:almostPosSemiDef} holds if and only if
\begin{equation}
	E_{11}\geq0 \quad\text{and}\quad E_{22}\geq0 \quad\text{and}\quad \big[ E_{12}<0 \;\implies\; \det E \geq 0 \big]\,,
\end{equation}
which, after some computation, can be stated as
\begin{align}
	\psi^\prime(I)\geq 0 \quad\text{and}\quad \Big[(\lambda_1^2+\lambda_2^2)\,\psi'(I) + 2(\lambda_1^2-\lambda_2^2)^2\,\psi''(I) < 0 \;&\implies\; (\lambda_1+\lambda_2)^2\, \Psi'(I) + 2(\lambda_1^2-\lambda_2^2)^2\, \Psi''(I) \geq 0 \Big]\,. \label{eq:detImplication}
\intertext{Since under the assumption $\psi^\prime(I)\geq 0$, the implication}
	(\lambda_1^2+\lambda_2^2)\,\psi'(I) + 2(\lambda_1^2-\lambda_2^2)^2\,\psi''(I) \geq 0 \;&\implies\; (\lambda_1+\lambda_2)^2\, \Psi'(I) + 2(\lambda_1^2-\lambda_2^2)^2\, \Psi''(I) \geq 0 \nonumber
\end{align}
holds in general, we conclude that condition \eqref{eq:detImplication} (and thus the Legendre-Hadamard ellipticity of $W$) is equivalent to
\[
	\psi^\prime(I)\geq 0 \quad\text{and}\quad (\lambda_1+\lambda_2)^2\, \Psi'(I) + 2(\lambda_1^2-\lambda_2^2)^2\, \Psi''(I) \geq 0\,,
\]
which we can write as
\begin{align}\label{ineqW3}
	\psi^\prime(I)&\geq 0\,, \qquad \psi^\prime(I)+2(\lambda_1-\lambda_2)^2\, \psi^{\prime\prime}(I)\geq 0\,. \nonumber
\end{align}
Finally, in terms of $I$, this can be expressed as
\begin{align}
	\psi^\prime(I)&\geq 0\,, \qquad \psi^\prime(I)+2(I-2)\, \psi^{\prime\prime}(I)\geq 0\,, \nonumber
\end{align}
which is the criterion for rank-one convexity given in Proposition \ref{propFosdick} and Proposition \ref{Abeyaratne}.
}
\end{document}

%% file: SLconvexityApproachVisualization.tex
\providecommand{\afrac}[2]{#1\!/\!#2}

\def\SLconVisDistX{3.5}
\def\SLconVisDistY{.7}

\tikzset{nodeStyle/.style={align=center,color=black,shape=rectangle,rounded corners=0.5ex,draw=black,thick,fill=blue!10}}
\tikzset{implication/.style={double,double equal sign distance,-implies, thick}}
\tikzset{equivalence/.style={double,double equal sign distance,implies-implies, thick}}

\node[above] at (0,1.4) {\framebox{\parbox{6.3cm}{\center \large%
	$\Winc\col\SL(2)\to\R\,,\quad F\mapsto\Winc(F)$\\%
	[.105cm]$\Wiso(F)\colonequals\Winc\left(\frac{F}{(\det F)^{\afrac1n}}\right)$}}};

\node[nodeStyle](roSL) at (-\SLconVisDistX,-\SLconVisDistY) {$\Winc$ rank-one convex on $\SL(2)$};
\node[nodeStyle](roGL) at (-\SLconVisDistX,\SLconVisDistY) {$\Wiso$ rank-one convex on $\GLp(2)$};
\node[nodeStyle](polyGL) at (\SLconVisDistX,\SLconVisDistY) {$\Wiso$ polyconvex on $\GLp(2)$};
\node[nodeStyle](polySL) at (\SLconVisDistX,-\SLconVisDistY) {$\Winc$ polyconvex on $\SL(2)$};

\draw[color=red] (roSL) edge[implication] node[pos=0.5,right]{\large\textbf{?}} node[pos=0.5,left]{(1)} (roGL);
\draw (roGL) edge[implication] node[pos=0.5,above]{(2)} (polyGL);
\draw (polyGL) edge[implication] node[pos=0.5,right]{(3)} (polySL);
\draw[color=blue] (roSL) edge[implication] node[pos=0.5,above]{\large\textbf{?}} node[pos=0.5,below]{(0)} (polySL);